\documentclass[12pt]{article}
\usepackage{latexsym}
\usepackage{amsmath, amsthm}
\usepackage{amsfonts}
\usepackage{mathrsfs}
\newtheorem{prop}{Proposition}[section]
\newtheorem{lem}{Lemma}[section]
\newtheorem{thm}{Theorem}[section]
\newtheorem{cor}{Corollary}[section]

\newtheorem{defn}{Definition}[section]

\newcommand{\rank}{\mathrm{rank}\,}

\newcommand{\Rs}{\mathbb{R}}

\newcommand{\G}{\overline{G} }

\newcommand{\N}{\mbox{null}}
\newcommand{\bz}{{\bf 0} }

\newcommand{\beq}{ \begin{equation} }
\newcommand{\eeq}{ \end{equation} }
\newcommand{\bt}{ \begin{tabular} }
\newcommand{\et}{ \end{tabular} }

\begin{document}

\bibliographystyle{plain}
\title{On the Colin de Verdi\`{e}re Graph Number and Penny Graphs }
\vspace{0.3in}
        \author{ A. Y. Alfakih
  \thanks{E-mail: alfakih@uwindsor.ca}
  \\
          Department of Mathematics and Statistics \\
          University of Windsor \\
          Windsor, Ontario N9B 3P4 \\
          Canada
}

\date{ \today}
\maketitle

\noindent {\bf AMS classification:} 05C50, 05C10, 05C62, 15B48.

\noindent {\bf Keywords:} The Colin de Verdi\`{e}re graph number, penny graphs, 
Euclidean distance matrices.
\vspace{0.1in}

\begin{abstract}

The Colin de Verdi\`{e}re number of graph $G$, denoted by $\mu(G)$,
is  a spectral invariant of $G$ that is related to some of its 
topological properties. 
For example, $\mu(G) \leq 3$ iff $G$ is planar.   
A {\em penny graph} is the contact graph of equal-radii disks with 
disjoint interiors in the plane. In this note, we prove lower bounds 
on $\mu(G)$ when the complement $\G$ is a penny graph.  

\end{abstract}

\section{Introduction and the Main Result}

A {\em penny graph} is the contact graph of equal-radii disks with 
disjoint interiors in the plane. To be more precise, 
a graph $G$ on $n$ nodes is a penny graph if there exists a 
one-to-one correspondence between the nodes of $G$ and a set of $n$  
closed {\em unit-diameter} disks with disjoint interiors in the plane
such that two nodes of $G$ are adjacent if and only if their 
corresponding disks touch. For $i=1,\ldots,n$, let  
$p^i$ denote the center of the disk corresponding to node $i$. 
Then the set of points $\{ p^1, \ldots, p^n \}$ is called a
{\em realization} of $G$, and it immediately follows that  
the edge set of $G$ is given by 
\beq \label{defE} 
E(G)=\{ \{i,j\}: ||p^i-p^j||^2 = 1\}, 
\eeq
where $||x||$ denotes the Euclidean norm of $x$, i.e., $||x||=\sqrt{x^Tx}$. 
Moreover, the set of missing edges of $G$, or equivalently, the edge 
set of the complement graph $\G$, is given by  
\beq \label{defEb} 
E(\G)=\{ \{i,j\}: ||p^i-p^j||^2 > 1\}.  
\eeq
As a result, penny graphs are also called {\em minimum-distance graphs}. 

A graph is {\em planar} if it can be drawn in the plane with its 
edges intersecting at their nodes only.  
Evidently, penny graphs are planar. A planar graph is said to be
{\em outerplanar} if all of its nodes lie on the boundary of the
outer face of its drawing.   

Let $A_G$ denote the generalized adjacency matrix of graph $G$, where
the 1's are replaced by arbitrary positive numbers and the diagonal
entries are arbitrary. 
The Colin de Verdi\`{e}re number of $G$, denote by $\mu(G)$
\cite{col90,col93}, is more or less the maximum multiplicity of     
the second largest eigenvalue, under a certain nondegeneracy assumption,
 of $A_G$. The precise definition of $\mu(G)$ is given in 
Subsection \ref{secCdV} below. 
Surprisingly, $\mu(G)$ is related to several topological properties of
$G$. For example, $\mu(G)$ is minor monotone, i.e., 
if $H$ is a minor of $G$, then $\mu(H) \leq \mu(G)$.    
Also, the planarity of $G$ can be characterized in terms of $\mu(G)$
as given in the following theorem. 

\begin{thm}[Colin de Verdi\`{e}re \cite{col90}]  
Let $G$ be a connected graph. Then
\begin{enumerate}
\item $\mu(G) \leq 1$ iff $G$ is a path.
\item $\mu(G) \leq 2$ iff $G$ is outerplanar.
\item $\mu(G) \leq 3$ iff $G$ is planar.
\end{enumerate}
\end{thm}

The results on $\mu(G)$ which are most relevant to this note
are given in the following two theorems.

\begin{thm}[Kotlov et al \cite{klv97}]  
Let $G$ be a graph on $n$ nodes. Assume that $G$ has no twins, i.e.,
no two nodes with same set of neighbors. Then  
\begin{enumerate}
\item if $\mu(\G) \geq n-3$, then $G$ is outerplanar.
\item if $\mu(\G) \geq n-4$, then $G$ is planar.
\end{enumerate}
\end{thm}

\begin{thm}[Kotlov et al \cite{klv97}] \label{klvthm} 
Let $G$ be a graph on $n$ nodes. Then 
\begin{enumerate}
\item if $G$ is a path, or a disjoint union of paths, 
                                   then $\mu(\G) \geq n-3$.
\item if $G$ is outerplanar, then $\mu(\G) \geq n-4$.
\item if $G$ is planar, then $\mu(\G) \geq n-5$.
\end{enumerate}
\end{thm}

The following theorem is the main result of this note.

\begin{thm} \label{mainthm} 
Let $G$ be a penny graph on $n$ nodes, where $n \geq 5$.  Then  
\begin{enumerate}
\item if $G$ is a path, a disjoint union of paths, or a cycle,  
then $\mu(\G) \geq  n-3$.   
\item Otherwise, $\mu(\G) \geq  n-4$.   
\end{enumerate}
\end{thm}

It is interesting to contrast Theorems \ref{klvthm} and \ref{mainthm}
since penny graphs are planar and since paths and cycles are trivially
outerplanar. Furthermore, the following remark is worth pointing out.
Let $K_n$ denote the complete graph on $n$ nodes and assume that $n \geq 3$. 
Then $\mu(K_n)= n-1$ and $\mu(G) \leq n-2$ for any graph $G$ on $n$ nodes
that is different from $K_n$ \cite{hls99}. 

The proof of Theorem \ref{mainthm}, which is based on the theory of
Euclidean distance matrices, is presented in Section 3. The  
necessary background for the proof is presented in Section 2.

\section{Preliminaries}

We begin this section by collecting the notation used throughout this note.
$E(G)$ denotes the edge set of graph $G$, while 
$E(\G)$ denotes the edge set of the
complement graph $\G$. We denote by $e$ and $E$, respectively, the 
$n$-vector and the $n \times n$ matrix of all 1's. 
$\bz$ denotes the zero vector or the zero matrix of appropriate dimension.
 $\N(A)$ denotes
the null space of matrix $A$ and $A_j$ denotes the $j$th column of $A$. 
Finally, $||x||$ denotes the Euclidean norm of $x$. 

\subsection{The Colin de Verdi\`{e}re number of graphs} 
\label{secCdV}

The {\em corank} of a matrix $M$ is the dimension of its null space.
Let $G$ be an undirected graph on $n$ nodes. A {\em $G$-matrix $A$} is 
an $n \times n$ symmetric matrix such that $a_{ij}=0$ for all  
$\{i,j\} \in E(\G)$. 
In addition, if $a_{ij} <0$ for all $\{i,j\} \in E(G)$, then  
$A$ is said to be {\em well signed}. 
Note that there is no condition on the diagonal entries of $A$. 

\begin{defn} \label{defCdV} 
Let $G$ be a connected graph on $n$ nodes. The 
Colin de Verdi\`{e}re number of $G$, denoted by $\mu(G)$, is the  
the maximum corank of an $n \times n$ matrix $M$ that satisfies the
following conditions: 
\begin{description} 
\item[M1:] $M$ is a well-signed $G$-matrix. 
\item[M2:] $M$ has exactly one negative eigenvalue. 
\item[M3:] There does not exist a nonzero $\G$-matrix $X$ whose diagonal
entries are all 0's such that $MX=\bz$. 
\end{description} 
\end{defn}
Condition M3 \cite{hls95} is one of several 
equivalent formulations of the {Strong Arnold Property (SAP)}.  
Any matrix that satisfies Conditions M1, M2 and M3 is called a  
{\em Colin de Verdi\`{e}re matrix} of graph $G$.   
A Colin de Verdi\`{e}re matrix $M$ of $G$ such that corank ($M$) = $\mu(G)$   
is called {\em optimal}.

The definition of $\mu(G)$ can be extended to disconnected graphs \cite{hls99}. 
Let $G_1, \ldots, G_k$ be the connected components of $G$ and assume
that $G$ has at least one edge. Then 
\[
\mu(G) = \max \{ \mu(G_1), \dots, \mu(G_k)\}. 
\] 

It is easy to show that $\mu(K_1)=0$ and $\mu(\overline{K_n}) = 1$ 
if $n \geq 2$. Also $\mu(K_n)=n-1$ and for any graph $G$ on $n$ nodes
($n \geq 3$) such that $G \neq K_n$, we have $\mu(G) \leq n-2$. Note that 
$\mu(K_2)=\mu(\overline{K_2})=1$.  
For a comprehensive survey of $\mu(G)$, see the paper \cite{hls99}
and the recent book \cite{lov18}.  

\subsection{Penny Graphs}

As was mentioned earlier, penny graphs are obviously planar. 
Whereas a planar graph
on $n$ nodes can have at most $3n-6$ edges, a penny graph can have at most
$3n-\sqrt{12n-3}$ edges \cite{har74}. Furthermore,
unlike planar graphs which can be recognized in linear 
time \cite{ht74}, the problem of recognizing penny graphs 
is NP-hard \cite{ew96}.
      
Evidently, the maximum degree of any node of a penny graph
is 6 since the kissing number in the plane is 6. Also, it is easy
to show that a penny graph cannot have $K_4$ or $K_{2,3}$ as subgraphs. 
In the remainder of this subsection, we present some properties
of penny graphs which are relevant to this note.

Let $k$ be a nonnegative integer. 
A graph $G$ is said to be {\em $k$-degenerate}
if every induced subgraph of $G$ has a node of degree at most $k$. 
The following are easy observations. 

\begin{prop}[\cite{lw70}] \label{prop1}  
Let $G$ be a graph on $n$ nodes. Then 
$G$ is $k$-degenerate if and only
if the nodes of $G$ can be ordered, say $v_1,\ldots,v_n$, such that
for all $i=1,\ldots,n$, 
the degree of $v_i$ in the subgraph induced by nodes $\{v_i,\ldots,v_n\}$ is 
at most $k$.  In other words, graph $G$ is $k$-degenerate iff
it can be reduced to a single node
by the successive removal of nodes of degree at most $k$. 
\end{prop}

\begin{prop}[\cite{epp18}]  \label{prop2} 
Let $G$ be a penny graph. Then $G$ is $3$-degenerate.
\end{prop}

Proposition \ref{prop2} follows since each vertex of the convex hull
of any subgraph of a penny graph has degree at most 3. 
The following technical lemmas easily follow from the geometry of penny graphs.

\begin{lem} \label{rholem}
Let $\{p^1, \ldots, p^n\}$ be a realization of a penny graph $G$; and
assume that $p^1, \ldots, p^n$ lie on a circle of radius $\rho$. 
If $n \geq 5$, then $\rho^2 > 1/2$.
\end{lem}

\begin{proof}
Let $\theta$ be the central angle formed by two touching disks. Then
obviously, $\theta \leq 2 \pi/5$. Hence, 
\[ \rho^2 \geq \frac{1} {2(1-\cos \theta)} \geq 
\frac{1} { 2(1-\cos (2 \pi/5))} > \frac{1}{2}. 
\] 
\end{proof}

\begin{lem} \label{Gblem}
Let $G$ be a penny graph on $n$ nodes, where $n \geq 5$. If 
the complement graph $\G$ is not connected, then $n=5,6$ or $7$; and 
$G$ is realized by one disk touching, respectively, $4, 5$ or $6$ other disks. 
\end{lem}

\begin{proof}
Assume that $\G$ is not connected. If $\G$ has 4 or more connected
components, then it is easy to see that $G$ has a $K_4$ as a subgraph,
a contradiction. Similarly, if $\G$ has 3 connected
components, then  $G$ has a $K_{2,3}$ as a subgraph, also a contradiction. 
Hence, $\G$ has 2 connected components. Now if one of these components
has 2 or more nodes and the other component has 3 or more nodes, then  
 $G$ has a $K_{2,3}$ as a subgraph, again a contradiction. 
Therefore, $\G$ must have one isolated node, i.e.,
we must have one disk, corresponding to this isolated node,
touching the disks corresponding to all other nodes. 
The result follows since the kissing number in the plane is 6, and  
since $n \geq 5$. 
\end{proof}

\subsection{Euclidean Distance Matrices}

As was mentioned earlier, the proof of Theorem \ref{mainthm} is based
on the theory of Euclidean distance matrices (EDMs). 
In this subsection, we present the results of EDMs that are most relevant
to this note. 
For a comprehensive treatment see the monograph \cite{alf18m}.

An $n \times n$ matrix $D$ is called a {\em Euclidean distance matrix (EDM)}
if there exist points $p^1,\ldots,p^n$ in some Euclidean space such that
\[
d_{ij} = ||p^i - p^j||^2 \mbox{ for all } i,j=1,\ldots,n.
\] 
The points $p^1,\ldots,p^n$ are called the {\em generating points} of $D$ and 
the dimension of their affine span is called the {\em embedding dimension}
of $D$. Obviously, an EDM $D$ is symmetric with zero diagonal and nonnegative
offdiagonal entries.  Let $e$ denote the vector of all 1's in
$\Rs^n$, and let $V$ be an $n \times (n-1)$ matrix such that 
$Q=[e/\sqrt{n} \;\; V]$ is an $n \times n$ orthogonal matrix. Then 
we have the following well-known characterization of EDMs. 

\begin{thm} \label{EDMthm}[ \cite{gow85,sch35,yh38}] 
Let $D$ be an $n \times n$ symmetric 
matrix of zero diagonal. Then $D$ is an EDM if and only if $V^T (-D) V$ 
is positive semidefinite; in which case, the embedding dimension of $D$ 
is equal to the rank of $V^TDV$. 
\end{thm}
That is, a symmetric matrix $D$ with zero diagonal is an EDM iff $D$
is negative semidefinite on $e^{\perp}$, the orthogonal complement of
$e$ in $\Rs^n$.
EDMs have the nice property that $e$ lies in the column space
of every nonzero EDM $D$ \cite{gow85}. i.e., for any EDM $D \neq \bz$, 
there exists a vector $w$ such that $Dw=e$. 

An EDM $D$ is said to be {\em spherical} if its generating points lie
on a sphere. Otherwise, it is said to be {\em nonspherical}. 
If the generating points of $D$ lie on a sphere of radius $\rho$,
we will refer to $\rho$ as the radius of $D$. 
Spherical and nonspherical EDMs have many different characterizations.
The most relevant for our purposes are those given in the following
two theorems.

\begin{thm} [\cite{gow82,gow85,neu81,thw96}] \label{sEDM}   
Let $D$ be a nonzero $n \times n$ EDM of embedding dimension $r$ 
and let $Dw=e$. If $r=n-1$, then $D$ is spherical; and if $r \leq n-2$, then   
the following statements are equivalent:
\begin{enumerate}
\item $D$ is spherical of radius $\rho$.
\item $\rank (D) = r+ 1$.
\item $e^Tw > 0$ and $\rho^2 = 1/(2 e^T w)$. 
\item There exists a scalar $\beta$ such that $\beta E - D$ is
      positive semidefinite; moreover, $\beta=2 \rho^2$ is the smallest 
      such scalar.  
\end{enumerate}
\end{thm}

\begin{thm} [Gower  \cite{gow82,gow85}] \label{nsEDM}   
Let $D$ be a nonzero EDM of embedding dimension $r$ and let $Dw=e$.
Then the following statements are equivalent:
\begin{enumerate}
\item $D$ is nonspherical.
\item $\rank (D) = r+ 2$.
\item $e^Tw = 0$. 
\end{enumerate}
\end{thm}

The following lemmas will be needed when dealing  
with the Strong Arnold Property.

\begin{lem} \label{DMlem}
Let $D$ be an EDM and let $M=E-D$. Then $\N(D) \subseteq \N(M)$.
\end{lem}

\begin{proof}
Let $x \in \N(D)$. Then $Mx=e^Tx \; e$. But $e^T x = 0$ since $e$ lies
in the column space of $D$.
Therefore $x \in \N(M)$ and thus $\N(D) \subseteq \N(M)$.
\end{proof}

\begin{lem} \label{MDlem}
Let $D$ be an EDM and let $M=E-D$. Assume that $D$ is nonspherical or
spherical with radius $\rho \neq 1/\sqrt{2}$. 
Then $\N(M) \subseteq \N(D)$.
\end{lem}

\begin{proof}
Let $x \in \N(M)$. Then $Dx=e^Tx \; e$. Let $Dw=e$. Then
$w^T D x = e^Tx \; e^Tw$. Hence, $e^Tx(1-e^Tw)=0$.
Now if $D$ is nonspherical, then, it follows from  Theorem \ref{nsEDM}
that $e^Tw=0$ and hence $e^T x = 0$. 
On the other hand, if  
$D$ is spherical with $\rho^2 \neq 1/2$, then it follows from
part 3 of Theorem \ref{sEDM} that $e^Tw \neq 1$ and hence again $e^T x =0$. 
Consequently, $Dx=0$ and thus $\N(M) \subseteq \N(D)$.
\end{proof}

The following corollary immediately follows from 
Lemmas \ref{DMlem} and \ref{MDlem} and  Lemma \ref{rholem}. 

\begin{cor} \label{NDNMcor}
Let $D$ be an $n \times n$ EDM, where $n \geq 5$, 
and let $M=E-D$. Then $\N(M)=\N(D)$.
\end{cor}

\begin{lem} \label{3lilem} 
Let $D$ be an $n \times n$ EDM, where $n \geq 3$, and assume that 
all of its offdiagonal entries are positive. 
Then any 3 columns of $D$ are linearly independent.  
\end{lem}

\begin{proof}
By way of contradiction, assume that the columns $D_{i_1}, D_{i_2}$ 
$D_{i_3}$ of $D$ are linearly dependent, and let $\tilde{D}$ denote the  
$3 \times 3$ principal submatrix of $D$ induced by 
the indices $\{i_1,i_2,i_3\}$. Then the EDM 
$\tilde{D}$ is singular. Let $p^{i_1}, p^{i_2}, p^{i_3}$ denote the
generating points of $\tilde{D}$. Then, by our assumption,   
$p^{i_1}, p^{i_2}, p^{i_3}$ are distinct.  

Now if $p^{i_1}, p^{i_2}, p^{i_3}$ are collinear, then 
$\tilde{D}$ is nonspherical of embedding dimension 1, and hence,
by Theorem \ref{nsEDM}, $\rank(\tilde{D})=3$, a contradiction.  
On the other hand, if $p^{i_1}, p^{i_2}, p^{i_3}$ 
are not collinear, then 
$\tilde{D}$ has embedding dimension 2, and hence,
$\tilde{D}$ is spherical. Therefore, by   
Theorem \ref{sEDM}, $\rank(\tilde{D})=3$, also a contradiction.  
Hence, the result follows. 
\end{proof}

\section{Proof of Theorem \ref{mainthm}}

To prove Theorem \ref{mainthm}, it suffices to exhibit 
a Colin de Verdi\`{e}re matrix
of the complement graph $\G$ of the appropriate corank. To this end, 
every  realization $\{p^1,\ldots,p^n\}$ of a penny graph $G$
defines an EDM $D = (d_{ij} = ||p^i - p^j||^2)$.  
Hence, by Equations (\ref{defE}) and (\ref{defEb}), it immediately follows 
that all the offdiagonal entries of $D$ are $\geq 1$. To be more precise, 
let 
\beq \label{defM}
M = E - D, 
\eeq   
then 
\beq
M_{ij} 
\left\{ \begin{array}{ll} = 1 & \mbox{ if } i=j, \\   
                         =  0 & \mbox{ for all } \{i,j\} \in E(G), \\  
                           < 0 & \mbox{ for all } \{i,j\} \in E(\G). 
         \end{array} \right. 
\eeq 
Therefore, 
$M$ is a well-signed $\G$-matrix, and thus $M$ satisfies Condition
M1 of Definition \ref{defCdV} for the complement graph $\G$. 
To prove that $M$ is the desired Colin de Verdi\`{e}re matrix for 
$\G$, we need to show that matrix $M$ also satisfies 
Conditions M2 and M3 for $\G$. This we do next.

\subsection{Proof that $M$ Satisfies Condition M2}

It follows from Theorem \ref{EDMthm} that 
$V^T (-D) V =  V^T M V$ is an $(n-1) \times (n-1)$
positive semidefinite matrix of rank 2. 
Let $Q=[e/\sqrt{n} \;\; V]$. Then matrices $M$ and 
\[
Q^T M Q = \left[ \begin{array}{cc} 
                      e^TM e/n & e^TMV/ \sqrt{n} \\
                      V^TM e/ \sqrt{n} & V^T M V 
           \end{array} \right] 
\] 
are similar. Let $\lambda_1$ and $\lambda_2$ denote, respectively,
 the smallest and the second smallest eigenvalues of $M$. 
Thus, it follows from the eigenvalue interlacing theorem that 
$\lambda_2 = 0$ and $\lambda_1 \leq 0$ since $V^TMV$ is positive semidefinite   
of rank 2.  Recall that we assume that $n \geq 5$. 

Now if $\lambda_1=0$, then $M$ is positive semidefinite, then 
it follows from Theorem \ref{EDMthm} and part 4 of Theorem \ref{sEDM} that
$D$ is a spherical EDM with radius $\rho \leq 1/\sqrt{2}$. This
contradicts Lemma \ref{rholem} since we assume that $n \geq 5$. 
Therefore, $\lambda_1 < 0$ and thus $M$ satisfies Condition M2.

\subsection{Proof that $M$ Satisfies Condition M3}

Let $X$ be a $G$-matrix whose diagonal entries are all 0's and
let $MX=\bz$. Then it follows from Corollary \ref{NDNMcor} that
$DX=\bz$. Therefore, 
\beq \label{nsys} 
 \sum_{j: j \in N(i)} x_{ij}  D_j = \bz \mbox{ for all } i=1,\ldots, n,
\eeq 
where $N(i)$ denotes the set of nodes of $G$ that are adjacent to node $i$, 
and $D_j$ denotes the $j$th column of $D$. 
Now Propositions \ref{prop1} and \ref{prop2} imply that $G$ can be reduced
to a single node by the successive removal of nodes of degree at most 3.  
Therefore, by solving the $n$ systems of equations 
of (\ref{nsys}) in the same order as that of removing these nodes, 
we obtain that  
$X=\bz$ since, by Lemma \ref{3lilem}, any 3 columns of $D$ are linearly 
independent. Consequently, $M$ satisfies Condition M3 and as a result,
$M$ is a Colin de Verdi\`{e}re matrix of $\G$. 

\subsection{Establishing the Corank of $M$} 

Assume that $G$ is a path, a disjoint union of paths, or a cycle.  
Then $G$ has a realization whose corresponding EDM $D$ is spherical.
Then, by part 2 of Theorem \ref{sEDM} and Lemmas \ref{DMlem} and \ref{MDlem}, 
$\rank(M)=\rank(D)=3$. Hence, corank($M$) = $n-3$ and thus
$\mu(\G) \geq n-3$. Note that, by Lemma \ref{Gblem}, $\G$ is connected. 

Now assume that $G$ is not a path, a disjoint union of paths, or a cycle.  
Then $G$ has a realization whose corresponding EDM $D$ is nonspherical.
Assume that $\G$ is connected. Then, by part 2 of Theorem \ref{nsEDM},
$\rank(M) = 4$. Hence, corank($M$) = $n-4$ and thus $\mu(\G) \geq n-4$.
On the other hand, if $\G$ is not connected, then by Lemma \ref{Gblem},
$G$ consists of one isolated node and one connected component with $n-1$
nodes. Again, by Lemma \ref{Gblem}, this connected component has
a realization whose corresponding $(n-1) \times (n-1)$ EDM $D'$ is 
spherical. Thus, matrix $M' = E'- D'$ has   
rank $3$, where $E'$ is the matrix of all 1's of order $n-1$. 
Hence, corank($M'$)= $n-1-3=n-4$. Again $\mu(\G) \geq n-4$.




\end{document}